\documentclass[12pt, reqno]{amsart}
\usepackage{amsmath, amsthm, amscd, amsfonts, amssymb, graphicx, color}
\usepackage[bookmarksnumbered, colorlinks, plainpages]{hyperref}

\hypersetup{colorlinks=true,linkcolor=red, anchorcolor=green, citecolor=cyan, urlcolor=red, filecolor=magenta, pdftoolbar=true}
\newtheorem{theorem}{Theorem}[section]

\newtheorem{proposition}[theorem]{Proposition}
\newtheorem{corollary}[theorem]{Corollary}
\theoremstyle{definition}
\newtheorem{definition}[theorem]{Definition}

\theoremstyle{remark}

\numberwithin{equation}{section}
\input{tcilatex}

\begin{document}
\title[On the composition ideals ]{On the composition ideals of Lipschitz
mappings}
\author[K. Saadi ]{Khalil Saadi}
\address{University of M'sila, Laboratoire d'Analyse Fonctionnelle et G\'{e}%
om\'{e}trie des Espaces, 28000 M'sila, Algeria.}
\email{kh\_saadi@yahoo.fr}
\date{Received: xxxxxx; Revised: yyyyyy; Accepted: zzzzzz. }

\begin{abstract}
In this paper, we study some property of Lipschitz mappings which admit
factorization through an operator ideal. Lipschitz cross-norms has been
established from known tensor norms in order to represent certain classes of
Lipschitz mappings. Inspired by the definition of $p$-summing linear
operators, we derive a new class of Lipschitz mappings that is called
strictly Lipschitz $p$-summing.
\end{abstract}

\maketitle

\setcounter{page}{1}


\let\thefootnote\relax\footnote{%
Copyright 2016 by the Tusi Mathematical Research Group.}

\section{Introduction and preliminaries}

Let $X$ be a metric space and $E$ be a Banach space. Every Lipschitz mapping 
$T:X\longrightarrow E$ admits a factorization of the form%
\begin{equation}
T=\widehat{T}\circ \delta _{X},  \label{1.1}
\end{equation}%
where $\widehat{T}$ is the linearization of $T$ and $\delta _{X}$ is the
canonical embedding. Let $\mathcal{I}$ be an operator ideal, there is a
constructive method for defining new classes of Lipschitz mappings which it
consists in composing of linear operators of $\mathcal{I}$ and Lipschitz
mappings, the resulting space is denoted by $\mathcal{I}\circ Lip_{0}.$ This
technique is usually used to generate some ideals of multilinear mappings
and homogeneous polynomials (see \cite{2,11}). The study of the space $%
\mathcal{I}\circ Lip_{0}$ is well-motivated, many interesting spaces which
are resulting by this technique belong to famous classes of Lipschitz
mappings. This is the case, for example, of the spaces of Lipschitz-Cohen
strongly $p$-summing, Lipschitz compact, Lipschitz weakly compact, strongly
Lipschitz $p$-integral and strongly Lipschitz $p$-nuclear operators.
Moreover, the appearance of a linear operator and a Lipschitz mapping in the
formula \eqref{1.1} motivates us to investigate the connection between the
Lipschitz operator $T$ and its linearization. Given an operator ideal $%
\mathcal{I}$, by considering the correspondence $T\leftrightarrow \widehat{T}
$ we can obtain the following identification%
\begin{equation}
\mathcal{I}\circ Lip_{0}\left( X;E\right) =\mathcal{I}\left( \mathcal{F}%
\left( X\right) ;E\right) .  \label{1.2}
\end{equation}%
In this paper, our main objective is to derive and study new classes of
Lipschitz mappings which satisfy \eqref{1.2}$.$ We are interested to
represent these classes by using Lipschitz cross-norms, such norms have been
recently studied by Cabrera-Padilla et al., in \cite{3}$.$ First, we
established some relations between tensor norms which defined on tensor
product and Lipschitz cross-norms and we prove that every tensor norm
generates a Lipschitz cross-norm. So, if $\mathcal{I}$ is an operator ideal
which admits a representation through a tensor norm $\alpha ,$ i.e., 
\begin{equation*}
\mathcal{I}\left( E;F^{\ast }\right) =(E\widehat{\otimes }_{\alpha }F)^{\ast
},
\end{equation*}%
for every Banach spaces $E,F,$ then there is a Lipschitz cross-norm $\alpha
^{L}$\ for which the space $\mathcal{I}\circ Lip_{0}$ admits a Lipschitz
tensor representation, i.e.,%
\begin{equation*}
\mathcal{I}\circ Lip_{0}\left( X;E^{\ast }\right) =(X\widehat{\boxtimes }%
_{\alpha ^{L}}E)^{\ast },
\end{equation*}%
for every metric space $X$ and every Banach space $E$. Among our results, we
will investigate the Lipschitz cross-norms corresponding to Chevet-Saphar
norms. We will define a new concept in the category of Lipschitz operators,
that is Lipschitz strictly $p$-summing. The operators of this class have a
strong relationship with their linearizations for the concept of $p$%
-summing. Certain results and properties of this new class are obtained.%
\vspace{0.5cm}%

This paper is organized as follows. First, we recall some standard notations
which will be used throughout. In section 2, we define for a given operator
ideal $\mathcal{I}$, the class $\mathcal{I}\circ Lip_{0}$ of Lipschitz
mappings satisfying that their linearizations belong to $\mathcal{I}.$ Some
examples of classes of Lipschitz mappings which are represented by this
procedure are given. Section 3 contains the main results, we start by
studying Lipschitz cross-norms generated by tensor norms. Then, we consider
the Chevet-Saphar norms and we study the corresponding Lipschitz
cross-norms. Inspired by the definition of $p$-summing we introduce the
concept of Lipschitz strictly $p$-summing for which we prove that the
Lipschitz mapping $T$ is strictly $p$-summing if and only if its
linearization $\widehat{T}$\ is $p$-summing. This notion coincides with the
notions of $p$-summing and Lipschitz $p$-summing operators when we are
considering only linear operators.%
\vspace{0.5cm}%

Now, we recall briefly some basic notations and terminology. Throughout this
paper, the letters $E,F$ will denote Banach spaces and $X,Y$ will denote
metric spaces with a distinguished point (pointed metric spaces) which we
denote by $0$. Let $E$ be a Banach space and $n\in \mathbb{N}$. We denote by 
$l_{p}^{n}\left( E\right) ,$ $\left( 1\leq p\leq \infty \right) ,$ the space
of all sequences $\left( x_{i}\right) _{1\leq i\leq n}$ in $E$ equipped with
the norm%
\begin{equation*}
\left\Vert \left( x_{i}\right) _{1\leq i\leq n}\right\Vert _{l_{p}^{n}\left(
E\right) }=(\sum\limits_{1}^{n}\left\Vert x_{i}\right\Vert ^{p})^{\frac{1}{p}%
},
\end{equation*}%
and by $l_{p}^{n~\omega }\left( E\right) $ the space of all sequences $%
\left( x_{i}\right) _{1\leq i\leq n}$ in $E$ equipped with the norm%
\begin{equation*}
\left\Vert \left( x_{i}\right) _{1\leq i\leq n}\right\Vert _{l_{p}^{n~\omega
}\left( E\right) }=\underset{\left\Vert x^{\ast }\right\Vert _{E^{\ast }}=1}{%
\sup }(\sum\limits_{1}^{n}\left\vert x^{\ast }\left( x_{i}\right)
\right\vert ^{p})^{\frac{1}{p}}
\end{equation*}%
Let $X$ be a pointed metric space. We denote by $X^{\#}$ the Banach space of
all Lipschitz functions $f:X\longrightarrow \mathbb{R}$ which vanish at $0$
under the Lipschitz norm given by%
\begin{equation*}
Lip\left( f\right) =\sup \left\{ \frac{\left\vert f\left( x\right) -f\left(
y\right) \right\vert }{d\left( x,y\right) }:x,y\in X,x\neq y\right\} .
\end{equation*}%
We denote by $\mathcal{F}\left( X\right) $ the free Banach space over $X$,
i.e., $\mathcal{F}\left( X\right) $ is the completion of the space%
\begin{equation*}
AE=\left\{ \sum_{i=1}^{n}\lambda _{i}\delta _{\left( x_{i},y_{i}\right)
},\left( \lambda _{i}\right) _{i=1}^{n}\subset \mathbb{R},\text{ }\left(
x_{i}\right) _{i=1}^{n},\left( y_{i}\right) _{i=1}^{n}\subset X\right\} ,
\end{equation*}%
with the norm%
\begin{equation*}
\left\Vert m\right\Vert _{\mathcal{F}\left( X\right) }=\inf \left\{
\sum_{i=1}^{n}\left\vert \lambda _{i}\right\vert d\left( x_{i},y_{i}\right)
:m=\sum_{i=1}^{n}\lambda _{i}\delta _{\left( x_{i},y_{i}\right) }\right\} ,
\end{equation*}%
where the function $\delta _{\left( x,y\right) }:X^{\#}\rightarrow \mathbb{R}
$ is defined as follows%
\begin{equation*}
\delta _{\left( x,y\right) }\left( f\right) =f\left( x\right) -f\left(
y\right) .%
\vspace{0.5cm}%
\end{equation*}%
We have $\mathcal{F}\left( X\right) ^{\ast }=X^{\#}.$ For a general theory
of free Banach space see \cite{10,11,16}. Let $X$ be a metric space and $E$
be a Banach space, we denote by $Lip_{0}\left( X;E\right) $ the Banach space
of all Lipschitz functions $T:X\rightarrow E$ such that $T\left( 0\right) =0$
with pointwise addition and Lipschitz norm. Note that for any $T\in
Lip_{0}\left( X;E\right) $ there exists a unique linear map (linearization
of $T$) $\widehat{T}:\mathcal{F}\left( X\right) \longrightarrow E$ such that 
$\widehat{T}\circ \delta _{X}=T$ and $\left\Vert \widehat{T}\right\Vert
=Lip\left( T\right) ,$ i.e., the following diagram commutes%
\begin{equation}
\begin{array}{ccc}
X & \overset{T}{\longrightarrow } & E \\ 
\delta _{X}\downarrow  & \nearrow \widehat{T} &  \\ 
\mathcal{F}\left( X\right)  &  & 
\end{array}
\label{1.3}
\end{equation}%
where $\delta _{X}$ is the canonical embedding so that $\left\langle \delta
_{X}\left( x\right) ,f\right\rangle =\left\langle \delta _{\left( x,0\right)
},f\right\rangle =f\left( x\right) $ for $f\in X^{\#}.$ The Lipschitz
transpose map of a Lipschitz operator $T:X\rightarrow E$ is a linear
operator $T^{t}:E^{\ast }\rightarrow X^{\#}$ which is defined by%
\begin{equation*}
T^{t}\left( e^{\ast }\right) \left( x\right) =e^{\ast }\left( T\left(
x\right) \right) .
\end{equation*}%
We have%
\begin{equation*}
T^{t}=Q_{X}^{-1}\circ \widehat{T}^{\ast },
\end{equation*}%
where $Q_{X}$ is the isomorphism isometric between $X^{\#}$ and $\mathcal{F}%
\left( X\right) ^{\ast }$ such that%
\begin{equation*}
Q_{X}\left( f\right) \left( m\right) =m\left( f\right) ,\text{ for every }%
f\in X^{\#}\text{ and }m\in \mathcal{F}\left( X\right) .
\end{equation*}%
If $X$ is a Banach space and $T:X\rightarrow E$ is a linear operator, then
the corresponding linear operator $\widehat{T}$ is given by%
\begin{equation}
\widehat{T}=T\circ \beta _{X},  \label{1.4}
\end{equation}%
where $\beta _{X}:\mathcal{F}\left( X\right) \rightarrow X$ is linear
quotient map which verifies $\beta _{X}\circ \delta _{X}=id_{X}$ and $%
\left\Vert \beta _{X}\right\Vert \leq 1,$ see \cite[p 124]{13} for more
details about the operator $\beta _{X}$. Let $X$ be a metric space and $E$
be a Banach space, by $X\boxtimes E$ we denote the Lipschitz tensor product
of $X$ and $E$. This is the vector space spanned by the linear functional $%
\delta _{\left( x,y\right) }\boxtimes e$ on $Lip_{0}\left( X;E^{\ast
}\right) $ defined by%
\begin{equation*}
\delta _{\left( x,y\right) }\boxtimes e\left( f\right) =\left\langle f\left(
x\right) -f\left( y\right) ,e\right\rangle .
\end{equation*}%
Let $\alpha $ be a norm on $X\boxtimes E;$ $\alpha $ is called \textit{%
Lipschitz cross-norm} if it satisfies the condition%
\begin{equation*}
\alpha \left( \delta _{\left( x,y\right) }\boxtimes e\right) =d\left(
x,y\right) \left\Vert e\right\Vert .
\end{equation*}%
A Lipschitz cross-norm $\alpha $ is called \textit{dualizable} if given $%
f\in X^{\#}$ and $e^{\ast }\in E^{\ast }$, for all $\dsum\limits_{i=1}^{n}%
\delta _{\left( x_{i},y_{i}\right) }\boxtimes e_{i}\in X\boxtimes E,$ we
have 
\begin{equation*}
\left\vert \dsum\limits_{i=1}^{n}\left( f\left( x_{i}\right) -f\left(
y_{i}\right) \right) \left\langle e^{\ast },e_{i}\right\rangle \right\vert
\leq Lip\left( f\right) \left\Vert e^{\ast }\right\Vert \alpha \left(
\dsum\limits_{i=1}^{n}\delta _{\left( x_{i},y_{i}\right) }\boxtimes
e_{i}\right) .
\end{equation*}%
Every Lipschitz mapping $T:X\rightarrow E^{\ast }$ admits a linear
functional $\varphi _{T}$ defined on the Lipschitz tensor product $%
X\boxtimes E$ by%
\begin{equation*}
\varphi _{T}(\sum_{i=1}^{n}\delta _{\left( x_{i},y_{i}\right) }\boxtimes
e_{i})=\sum_{i=1}^{n}\left\langle f\left( x_{i}\right) -f\left( y_{i}\right)
,e_{i}\right\rangle .
\end{equation*}%
As in \cite{1}, a subclass $\mathcal{A}$ of $Lip_{0}$ is to said a \textit{%
normed (Banach) Lipschitz ideal} if for every pointed metric space $X$ and
every Banach space $E,$ the pair $\left( \mathcal{A}\left( X;E\right)
,\left\Vert .\right\Vert _{\mathcal{A}}\right) $ is a normed (Banach) space
and

\noindent $\left( a\right) $ For every $f\in X^{\#}$ and $e\in E$ the
Lipschitz operator $f\boxtimes e:X\rightarrow E$ defined by $f\boxtimes
e\left( x\right) =f\left( x\right) e$ is in $\mathcal{A}\left( X;E\right) $
and 
\begin{equation*}
\left\Vert f\boxtimes e\right\Vert _{\mathcal{A}}\leq Lip\left( f\right)
\left\Vert e\right\Vert .
\end{equation*}%
\noindent $\left( b\right) $ For all $T\in \mathcal{A}\left( X;E\right) $,
we have 
\begin{equation*}
Lip\left( T\right) \leq \left\Vert T\right\Vert _{\mathcal{A}}.
\end{equation*}%
\noindent $\left( c\right) $ \textit{Ideal property}: Let $Z$ be a metric
space and $F$ a Banach space. The composed operator $u\circ T\circ g$ is in $%
\mathcal{A}\left( Z;F\right) $ and 
\begin{equation*}
\left\Vert u\circ T\circ g\right\Vert _{\mathcal{A}}\leq \left\Vert
u\right\Vert \left\Vert T\right\Vert _{\mathcal{A}}Lip\left( g\right)
\end{equation*}%
for every $g\in Lip_{0}\left( Z;X\right) $, $T\in \mathcal{A}\left(
X;E\right) $ and $u\in \mathcal{B}\left( E;F\right) ,(\mathcal{B}\left(
E;F\right) $ is the Banach space of all linear operators from $E$ into $F)$.

\section{Lipschitz spaces generated by the composition method}

In this section, we apply composition ideals technique to generate new
classes of Lipschitz mappings. Given an operator ideal $\mathcal{I}$, let $X$
be a pointed metric space and $E$ be a Banach space. A Lipschitz operator $%
T\in Lip_{0}\left( X;E\right) $ is said to be of type $\mathcal{I\circ }%
Lip_{0}$ if there exist a Banach space $Z,$ a Lipschitz operator $L\in
Lip_{0}\left( X;Z\right) $ and a linear operator $u\in \mathcal{I}\left(
Z;E\right) $ such that the following diagram commutes%
\begin{equation*}
\begin{array}{ccc}
X & \overset{T}{\longrightarrow } & E \\ 
L\searrow &  & \nearrow u \\ 
& Z & 
\end{array}%
\end{equation*}%
On the other hand, $T=L\circ u.$ If $\left( \mathcal{I},\left\Vert
.\right\Vert _{\mathcal{I}}\right) $ is a normed (Banach) ideal, the space $%
\mathcal{I}\circ Lip_{0}\left( X;E\right) $ is a normed (Banach) Lipschitz
ideal with respect to the following norm%
\begin{equation*}
\left\Vert T\right\Vert _{\mathcal{I}\circ Lip_{0}}=\inf Lip\left( L\right)
\left\Vert u\right\Vert _{\mathcal{I}}.
\end{equation*}%
In \cite{1}$,$ the connection between the Lipschitz operators of $\mathcal{I}%
\circ Lip_{0}$ and their linearizations is established.

\begin{theorem}
\label{Theorem 2.1}\cite{1} \textit{Let} $\mathcal{I}\circ Lip_{0}$ \textit{%
be the space of Lipschitz mappings} \textit{generated by the normed operator
ideal} $\mathcal{I}$. \textit{The following properties are equivalent.}

$\left( 1\right) $ \textit{The Lipschitz operator }$T$\textit{\ belongs to} $%
\mathcal{I}\circ Lip_{0}\left( X;E\right) .$

$\left( 2\right) $ \textit{The linearization }$\widehat{T}$\textit{\ belongs
to }$\mathcal{I}\left( \mathcal{F}\left( X\right) ;E\right) .$

\textit{In this case we have }$\left\Vert T\right\Vert _{\mathcal{I}\circ
Lip_{0}}=\left\Vert \widehat{T}\right\Vert _{\mathcal{I}}$ \textit{and then }%
\begin{equation*}
\mathcal{I}\circ Lip_{0}\left( X;E\right) =\mathcal{I}\left( \mathcal{F}%
\left( X\right) ;E\right) \text{ \textit{holds isometrically}}.
\end{equation*}
\end{theorem}

\begin{proposition}
\label{Proposition 2.2}\textit{Let} $\mathcal{I}_{2}$, $\mathcal{I}_{2}$ 
\textit{be two operator ideals. Then,}

$\left( 1\right) $ \textit{If} $\mathcal{I}_{1}\circ Lip_{0}\left(
X;E\right) \subset\mathcal{I}_{2}\circ Lip_{0}\left( X;E\right) $ \textit{%
then,} $\mathcal{I}_{1}\left( \mathcal{F}\left( X\right) ;E\right) \subset%
\mathcal{I}_{2}\left( \mathcal{F}\left( X\right) ;E\right) .$

$\left( 2\right) $ \textit{If} $\mathcal{I}\circ Lip_{0}\left( X;E\right)
=Lip_{0}\left( X;E\right) $ \textit{then,} $\mathcal{I}\left( \mathcal{F}%
\left( X\right) ;E\right) =\mathcal{B}\left( \mathcal{F}\left( X\right)
;E\right) .$
\end{proposition}

\begin{proof}
$\left( 1\right) $ Let $u\in \mathcal{I}_{1}\left( \mathcal{F}\left(
X\right) ;E\right) $, then the Lipschitz operator $T=u\circ \delta
_{X}:X\longrightarrow E$ verifies%
\begin{equation*}
\widehat{T}=u.
\end{equation*}%
By Theorem \ref{Theorem 2.1}$,$ $T\in \mathcal{I}_{1}\circ Lip_{0}\left(
X;E\right) $ and hence $T\in \mathcal{I}_{2}\circ Lip_{0}\left( X;E\right) .$
Consequently $\widehat{T}=u\in \mathcal{I}_{2}\left( \mathcal{F}\left(
X\right) ;E\right) .$

$\left( 2\right) $ Let $u\in\mathcal{B}\left( \mathcal{F}\left( X\right)
;E\right) ,$ then $T=u\circ\delta_{X}\in Lip_{0}\left( X;E\right) .$ Hence%
\begin{equation*}
\widehat{T}=u\in\mathcal{I}\left( \mathcal{F}\left( X\right) ;E\right) .
\end{equation*}
\end{proof}

The next Proposition follows directly from the previous one.

\begin{proposition}
\label{Proposition 2.3}\textit{Let }$E$\textit{\ be a Banach space. The
following properties are equivalent.}

$\left( 1\right) $ $id_{E}\in\mathcal{I}\left( E;E\right) .$

$\left( 2\right) $ $\mathcal{I}\circ Lip_{0}\left( X;E\right) =Lip_{0}\left(
X;E\right) $ \textit{for every pointed metric space }$X.$
\end{proposition}

As in the linear case studied in \cite{1}\ we give the definition of
Lipschitz dual of a given operator ideal.

\begin{definition}
\cite{1} The Lipschitz dual of a given operator ideal $\mathcal{I}$ is
defined by%
\begin{equation*}
\mathcal{I}^{Lip_{0}-dual}\left( X;E\right) =\left\{ T\in Lip_{0}\left(
X;E\right) :T^{t}\in \mathcal{I}\left( E^{\ast };X^{\#}\right) \right\} .
\end{equation*}%
If $\left( \mathcal{I},\left\Vert .\right\Vert _{\mathcal{I}}\right) $ is a
normed (Banach) ideal, define%
\begin{equation*}
\left\Vert T\right\Vert _{\mathcal{I}^{Lip_{0}-dual}}=\left\Vert
T^{t}\right\Vert _{\mathcal{I}},
\end{equation*}%
then, the space $\mathcal{I}^{Lip_{0}-dual}\left( X;E\right) $ becomes a
normed (Banach) Lipschitz ideal.
\end{definition}

An operator ideal\ $\mathcal{I}$\ is symmetric if 
\begin{equation*}
\mathcal{I}=\mathcal{I}^{dual}=\left\{ u\in \mathcal{I}\left( F;G\right)
:u^{\ast }\in \mathcal{I}\left( G^{\ast };F^{\ast }\right) \right\} .
\end{equation*}%
If $\mathcal{I}$ is symmetric, we have the following coincidence between a
Lipschitz ideal and its dual.

\begin{proposition}
\label{Proposition 2.5}\textit{The following properties are equivalent.}

$\left( 1\right) $ $\mathcal{I}$ \textit{is symmetric.}

$\left( 2\right) $ $\mathcal{I}^{Lip_{0}-dual}\left( X;E\right) =\mathcal{%
I\circ}Lip_{0}\left( X;E\right) $ \textit{for every pointed metric space }$X$%
\textit{\ and Banach space }$E.$
\end{proposition}

Now, in the rest of this section we present some examples of classes of
Lipschitz mappings which generated by known operator ideals using the
composition method. In \cite{15}, we have introduced the class $\mathcal{D}%
_{p}^{L}$ of Lipschitz-Cohen strongly $p$-summing operators. Proposition 3.1
in \cite{15} asserts that $T:X\rightarrow E$ is Lipschitz-Cohen strongly $p$%
-summing if and only if its linearization $\widehat{T}$ is strongly $p$%
-summing. Combining with Theorem \ref{Theorem 2.1}, the class $\mathcal{D}%
_{p}^{L}$ can be interpreted in terms of the composition method as follows.

\begin{theorem}
\label{Theorem 2.6}\textit{Let }$X$\textit{\ be a pointed metric space and }$%
E$\textit{\ be a Banach space. Let }$1<p\leq \infty $ \textit{and let }$%
p^{\ast }$\textit{\ be its conjugate }$(\frac{1}{p}+\frac{1}{p^{\ast }}=1)$%
\textit{. We have}%
\begin{equation*}
\mathcal{D}_{p}^{L}\left( X;E\right) =\mathcal{D}_{p}\circ Lip_{0}\left(
X;E\right) =\Pi _{p^{\ast }}^{dual}\circ Lip_{0}\left( X;E\right) =\Pi
_{p^{\ast }}^{Lip_{0}-daul}\left( X;E\right) ,
\end{equation*}%
\textit{where }$\Pi _{p}$\textit{\ and }$\mathcal{D}_{p}$\textit{\ are the
classes of }$p$\textit{-summing and strongly }$p$-summing \textit{linear
operators, respectively.}
\end{theorem}

The classes of Lipschitz compact and weakly compact operators have been
introduced in \cite{14}$.$ By $Lip_{0\mathcal{K}}$ and $Lip_{0\mathcal{W}}$\
we denote the Banach Lipschitz ideals of Lipschitz compact and Lipschitz
weakly compact operators, respectively. In \cite[Proposition 2.1 and 2.2]{14}%
, a similar relation as in Theorem \ref{Theorem 2.1} has been established.

\begin{theorem}
\label{Theorem 2.7}\textit{Let }$X$\textit{\ be a pointed metric space and }$%
E$\textit{\ be a Banach space. Then}

$\left( 1\right) $ $Lip_{0\mathcal{K}}\left( X;E\right) =\mathcal{K}\circ
Lip_{0}\left( X;E\right) =\mathcal{K}^{dual}\circ Lip_{0}\left( X;E\right) .$

$\left( 2\right) $ $Lip_{0\mathcal{W}}\left( X;E\right) =\mathcal{W}\circ
Lip_{0}\left( X;E\right) =\mathcal{W}^{dual}\circ Lip_{0}\left( X;E\right) .$
\end{theorem}

A simple consequence of the linear result given in \cite{8}\ asserts that
the Banach space $F$ is reflexive if and only if, for every Banach space $G$
and linear operator $v:F\rightarrow G$, $v$ is weakly compact. We have the
next characterization.

\begin{theorem}
\label{Theorem 2.8}\textit{Let }$X$\textit{\ be a pointed metric space. The
following properties are equivalent.}

$\left( 1\right) $ \textit{The metric space }$X$\textit{\ is finite}.

$\left( 2\right) $ \textit{For all Banach space }$E$\textit{, we have} $%
Lip_{0\mathcal{W}}\left( X;E\right) =Lip_{0}\left( X;E\right) .$
\end{theorem}

\begin{proof}
$\left( 1\right) \Rightarrow\left( 2\right) :$ Immediate.

$\left( 2\right) \Rightarrow \left( 1\right) :$ Let $E$ be a Banach space
and $v:\mathcal{F}\left( X\right) \rightarrow E$ be a linear operator. We
will show that $v$ is weakly compact. By $\left( 2\right) ,$ the Lipschitz
operator $T=v\circ \delta _{X}:X\rightarrow E$ is Lipschitz weakly compact.
Hence, $\widehat{T}=v$ is weakly compact. As consequence, the space $%
\mathcal{F}\left( X\right) $ is reflexive. But by \cite[Theorem 1]{7} the
space $\mathcal{F}\left( X\right) $ is never reflexive if $X$ is an infinite
metric space.
\end{proof}

Now, we recall the definition of strongly $p$-nuclear operators introduced
in \cite{5}$.$

\begin{definition}
Let $1\leq p<\infty $. A Lipschitz operator $T\in Lip_{0}(X;E)$ is called a
strongly Lipschitz $p$-nuclear $\left( 1\leq p<\infty \right) $ if there
exist operators $A\in B\left( l_{p};E\right) ,$ $b\in Lip_{0}\left(
X;l_{\infty }\right) $ and a diagonal operator $M_{\lambda }\in B\left(
l_{\infty };l_{p}\right) $ induced by a sequence $\lambda \in l_{p}$ such
that the following diagram commutes%
\begin{equation*}
\begin{array}{ccc}
X & \overset{T}{\longrightarrow } & E \\ 
b\downarrow &  & \uparrow A \\ 
l_{\infty } & \overset{M_{\lambda }}{\longrightarrow } & l_{p}%
\end{array}%
\end{equation*}%
The triple $(A,b,\lambda )$ is called a strongly Lipschitz $p$-nuclear
factorization of $T$. We denote $\mathcal{N}_{p}^{SL}\left( X;E\right) $ the
Banach space of all strongly Lipschitz $p$-nuclear operators from $X$ into $%
E $ with the norm%
\begin{equation*}
\emph{sv}_{p}^{L}\left( T\right) =\inf \left\Vert A\right\Vert \left\Vert
M_{\lambda }\right\Vert Lip\left( b\right) ,
\end{equation*}%
where the infimum is taken over all the above factorizations. For Banach
spaces $F,G,$ we denote by $\mathcal{N}_{p}\left( F;G\right) $ the space of
all $p$-nuclear linear operators which admit a factorization as in the
Lipschitz case with the operator $b$ being linear.
\end{definition}

\begin{proposition}
\label{Proposition 2.10}\textit{Let }$1\leq p<\infty .$ \textit{The
Lipschitz operator }$T:X\rightarrow E$\textit{\ is strongly Lipschitz }$p$%
\textit{-nuclear if and only if its linearization }$\widehat{T}$\textit{\ is 
}$p$\textit{-nuclear. Consequently} 
\begin{equation*}
\mathcal{N}_{p}^{SL}\left( X;E\right) =\mathcal{N}_{p}\circ Lip_{0}\left(
X;E\right) .
\end{equation*}
\end{proposition}

\begin{proof}
Let $T$ be a strongly Lipschitz $p$-nuclear operator, we have 
\begin{equation*}
T=A\circ M_{\lambda}\circ b.
\end{equation*}
We use the Lipschitz factorization of $T$ and $b$ 
\begin{equation*}
\widehat{T}\circ\delta_{X}=A\circ M_{\lambda}\circ\widehat{b}\circ\delta_{X},
\end{equation*}
by the uniqueness of linearization we obtain%
\begin{equation*}
\widehat{T}=A\circ M_{\lambda}\circ\widehat{b},
\end{equation*}
hence, $\widehat{T}$ is $p$-nuclear. The converse is immediate.
\end{proof}

In analogy with the definition of strongly Lipschitz $p$-nuclear operator,
the authors in \cite{14} have introduced the definition of strongly
Lipschitz $p$-integral operator. In fact, the same definition has been
introduced in \cite{5}$.$ In the first definition, the authors have
considered a factorization in which the left operator is linear and the
right one is Lipschitz. In the second definition, the role of these
operators has been changed.

\begin{definition}
Let $1\leq p<\infty $. A Lipschitz operator $T\in Lip_{0}(X;E)$ is called
strongly Lipschitz $p$-integral if there exist a finite measure space $%
(\Omega ,\Sigma ,\mu )$, a bounded linear operator $A\in \mathcal{B}%
(L_{p}(\mu );E^{\ast \ast })$ and a Lipschitz operator $b\in
Lip_{0}(X;L_{\infty }(\mu ))$ such that the following diagram commutes%
\begin{equation*}
\begin{array}{ccc}
X & \overset{T}{\longrightarrow }E\overset{\mathcal{K}_{E}}{\longrightarrow }
& E^{\ast \ast } \\ 
b\downarrow &  & \uparrow A \\ 
L_{\infty }\left( \mu \right) & \overset{I_{\infty ,p}}{\longrightarrow } & 
L_{p}\left( \mu \right)%
\end{array}%
\end{equation*}%
where $I_{\infty ,p}:L_{\infty }(\mu )\rightarrow L_{p}(\mu )$ is the formal
inclusion operator. The triple $(A,b,\mu )$ is called a strongly Lipschitz $%
p $-integral factorization of $T$. We denote by $\mathfrak{I}_{p}^{SL}\left(
X;E\right) $ the Banach space of all strongly Lipschitz $p$-integral
operators from $X$ into $E$ with the norm%
\begin{equation*}
\emph{si}_{p}^{L}\left( T\right) =\inf Lip\left( b\right) \left\Vert
A\right\Vert .
\end{equation*}
\end{definition}

For Banach spaces $F,G$ we denote by $\mathfrak{I}_{p}\left( F;G\right) $
the space of all $p$-integral linear operators. Using the same argument in
the proof of Proposition \ref{Proposition 2.10}, we can prove the following.

\begin{proposition}
\label{Proposition 2.12}\textit{Let }$1\leq p<\infty .$ \textit{The
Lipschitz operator }$T:X\rightarrow E$\textit{\ is strongly Lipschitz }$p$%
\textit{-integral if and only if its linearization }$\widehat{T}$\textit{\
is }$p$\textit{-integral. Consequently} 
\begin{equation*}
\mathfrak{I}_{p}^{SL}\left( X;E\right) =\mathfrak{I}_{p}\circ Lip_{0}\left(
X;E\right)
\end{equation*}%
\textit{holds isometrically. If }$p=1$ \textit{we have}%
\begin{align*}
\mathfrak{I}_{1}^{SL}\left( X;E\right) & =\mathfrak{I}_{1}\circ
Lip_{0}\left( X;E\right) \\
& =\mathfrak{I}_{1}^{dual}\circ Lip_{0}\left( X;E\right) .
\end{align*}
\end{proposition}

As in the linear case, we give a factorization result for strongly Lipschitz 
$p$-nuclear operators. For the proof, we use the linearization operators and
the result \cite[Theorem 5.27]{9}$.%
\vspace{0.5cm}%
$

\begin{theorem}
\label{Theorem 2.13}\textit{Let }$1\leq p<\infty .$\textit{\ A Lipschitz
operator }$T:X\rightarrow E$\textit{\ is strongly Lipschitz }$p$\textit{%
-nuclear if and only if there exist a Banach space }$Z$\textit{, a compact
linear operator }$v:Z\rightarrow E$\textit{\ and a strongly Lipschitz }$p$%
\textit{-integral operator }$L:X\rightarrow Z$\textit{\ such that }%
\begin{equation*}
T=v\circ L.
\end{equation*}%
\textit{In this case}%
\begin{equation*}
\emph{sv}_{p}^{L}\left( T\right) =\inf \left\Vert v\right\Vert \emph{si}%
_{p}^{L}\left( L\right) .
\end{equation*}
\end{theorem}

\begin{proof}
Let $T:X\rightarrow E$ be a strongly Lipschitz $p$-nuclear operator, then $%
\widehat{T}:\mathcal{F}\left( X\right) \rightarrow E$ is $p$-nuclear.
Theorem 2.27 in \cite{9} asserts that there exist a Banach space $Z$, a
compact operator linear $v:Z\rightarrow E$ and a $p$-integral operator $w:%
\mathcal{F}\left( X\right) \rightarrow Z$ such that\ $\widehat{T}=v\circ w.$
Then 
\begin{equation*}
\widehat{T}\circ \delta _{X}=v\circ w\circ \delta _{X}\Rightarrow T=v\circ L
\end{equation*}%
where $L=w\circ \delta _{X}$ which is strongly Lipschitz $p$-nuclear by
Proposition \ref{Proposition 2.13}$.$

Conversely, suppose that 
\begin{equation*}
T=v\circ L
\end{equation*}%
where $v$ is compact operator and $L$ is strongly Lipschitz $p$-integral,
then%
\begin{equation*}
\begin{array}{ccc}
X & \overset{T}{\rightarrow } & E \\ 
\delta _{X}\downarrow &  & \uparrow v \\ 
\mathcal{F}\left( X\right) & \overset{\widehat{L}}{\rightarrow } & Z%
\end{array}%
\end{equation*}%
i.e., $\widehat{T}=v\circ \widehat{L},$ with $\widehat{L}$ is $p$-integral.
So, $\widehat{T}=v\circ \widehat{L}$ is $p$-nuclear and then $T$ is strongly
Lipschitz $p$-nuclear.
\end{proof}

\section{Main results}

\subsection{Results on Lipschitz tensor product}

Let $F,G$ be two Banach spaces. We denote by $F\otimes G$ its algebraic
tensor product. There are various tensor norms that may be defined on the
tensor product $F\otimes G$. If we consider a pointed metric space $X$ and a
Banach space $E$, there is an attempt to generalize the definition of tensor
product to the category of metric spaces. The authors in \cite{3}\ have
studied the space $X\boxtimes E$ which called \textit{Lipschitz tensor
product. }Some Lipschitz cross-norms have been defined on this space\textit{.%
} In this section we provide to give some relations between Lipschitz
cross-norms and tensor norms. In the sequel, we will use the terminology of
Lipschitz cross-norms for norms defined on $X\boxtimes E$ and tensor norms
for norms defined on $F\otimes G.$

\begin{theorem}
\label{Theorem 3.1}\textit{Every tensor norm }$\alpha $\textit{\ generates a
dualizable Lipschitz cross-norm }$\alpha ^{L}$\textit{\ such that for all
pointed metric space }$X$\textit{\ and Banach space }$E$\textit{\ we have}%
\begin{equation}
\alpha ^{L}(\sum_{i=1}^{n}\delta _{\left( x_{i},y_{i}\right) }\boxtimes
e_{i})=\alpha (\sum_{i=1}^{n}\delta _{\left( x_{i},y_{i}\right) }\otimes
e_{i}),  \label{3.1}
\end{equation}%
\textit{where} $\sum_{i=1}^{n}\delta _{\left( x_{i},y_{i}\right) }\boxtimes
e_{i}\in X\boxtimes E$. \textit{In this case, the linear map} $\Phi
:X\boxtimes _{\alpha ^{L}}E\rightarrow \mathcal{F}\left( X\right) \otimes
_{\alpha }E$ \textit{defined by}%
\begin{equation*}
\Phi (\sum_{i=1}^{n}\delta _{\left( x_{i},y_{i}\right) }\boxtimes
e_{i})=\sum_{i=1}^{n}\delta _{\left( x_{i},y_{i}\right) }\otimes e_{i},
\end{equation*}%
\textit{is well-defined and is an isometry.}
\end{theorem}

\begin{proof}
Let $\alpha $ be a tensor norm. The properties of the norm $\alpha ^{L}$
have been inherited from that of $\alpha .$ Let $X$ be a pointed metric
space and $E$ a Banach space. Let $x,y\in X$ and $e\in E$. Then%
\begin{align*}
\alpha ^{L}\left( \delta _{\left( x,y\right) }\boxtimes e\right) & =\alpha
\left( \delta _{\left( x,y\right) }\otimes e\right) =\left\Vert \delta
_{\left( x,y\right) }\right\Vert \left\Vert e\right\Vert \\
& =d\left( x,y\right) \left\Vert e\right\Vert .
\end{align*}%
So, $\alpha ^{L}$ is Lipschitz cross-norm. Let $f\in X^{\#}\left( =\mathcal{F%
}\left( X\right) ^{\ast }\right) $ and $e^{\ast }\in E^{\ast }.$ We have%
\begin{align*}
\left\vert \dsum\limits_{i=1}^{n}\left( f\left( x_{i}\right) -f\left(
y_{i}\right) \right) \left\langle e^{\ast },e_{i}\right\rangle \right\vert &
=\left\vert f\otimes e^{\ast }(\dsum\limits_{i=1}^{n}\delta _{\left(
x_{i},y_{i}\right) }\otimes e_{i})\right\vert \\
& \leq Lip\left( f\right) \left\Vert e^{\ast }\right\Vert \alpha
(\dsum\limits_{i=1}^{n}\delta _{\left( x_{i},y_{i}\right) }\otimes e_{i}) \\
& \leq Lip\left( f\right) \left\Vert e^{\ast }\right\Vert \alpha
^{L}(\dsum\limits_{i=1}^{n}\delta _{\left( x_{i},y_{i}\right) }\boxtimes
e_{i}).
\end{align*}%
Then $\alpha ^{L}$ is dualizable. Now, it is easy to show that $\Phi $ is
linear. Let $u=\sum_{i=1}^{n}\delta _{\left( x_{i},y_{i}\right) }\boxtimes
e_{i}=0,$ we will show that $\Phi \left( u\right) =0.$ Indeed, let $f\in 
\mathcal{F}\left( X\right) ^{\ast }$ and $e^{\ast }\in E^{\ast }.$ Since $%
\mathcal{F}\left( X\right) ^{\ast }=X^{\#}$ we have by \cite[Proposition 1.6]%
{3}%
\begin{equation*}
\sum_{i=1}^{n}\left( f\left( x_{i}\right) -f\left( y_{i}\right) \right)
e^{\ast }\left( e_{i}\right) =0.
\end{equation*}%
Then%
\begin{equation*}
\sum_{i=1}^{n}f\left( \delta _{\left( x_{i},y_{i}\right) }\right) e^{\ast
}\left( e_{i}\right) =0.
\end{equation*}%
So, $\Phi \left( u\right) =0$ thus tell us that $\Phi $ is well defined. Let 
$u=\sum_{i=1}^{n}\delta _{\left( x_{i},y_{i}\right) }\boxtimes e_{i}\in
X\boxtimes _{\alpha ^{L}}E.$ By \eqref{3.1} we have%
\begin{align*}
\alpha (\Phi (u))& =\alpha (\sum_{i=1}^{n}\delta _{\left( x_{i},y_{i}\right)
}\otimes e_{i}) \\
& =\alpha ^{L}\left( u\right) .
\end{align*}%
then $\Phi $ is an isometry.
\end{proof}

Since $\Phi$ is a linear isometry, its range $\Phi\left( X\boxtimes
_{\alpha^{L}}E\right) $ is closed. On the other hand, the tensors of the
form $\sum_{i=1}^{n}\delta_{\left( x_{i},y_{i}\right) }\otimes e_{i}$ are
dense in $\mathcal{F}\left( X\right) \widehat{\otimes}_{\alpha}E.$ This
shows that the range $\Phi\left( X\boxtimes_{\alpha^{L}}E\right) $ is dense
in $\mathcal{F}\left( X\right) \widehat{\otimes}_{\alpha}E$ and thus $X%
\widehat{\boxtimes}_{\alpha^{L}}E$ is isometrically isomorphic to $\mathcal{F%
}\left( X\right) \widehat{\otimes}_{\alpha}E$.%
\vspace{0.5cm}%

\begin{corollary}
\label{Corollary 3.2}\textit{For every pointed metric space }$X$ \textit{and
Banach space} $E$ \textit{we have,}%
\begin{equation}
X\widehat{\boxtimes }_{\alpha ^{L}}E=\mathcal{F}\left( X\right) \widehat{%
\otimes }_{\alpha }E  \label{3.2}
\end{equation}%
\textit{holds isometrically.}
\end{corollary}

As a consequence of Theorem \ref{Theorem 2.1} and Corollary \ref{Corollary
3.2}, we get the following result.

\begin{corollary}
\label{Corollary 3.3}\textit{Let} $\mathcal{I}\circ L$ \textit{be a
Lipschitz ideal generated by the operator ideal} $\mathcal{I}$. \textit{%
Suppose that }$\mathcal{I}$ \textit{can be interpreted through a tensor
product, i.e., there is a tensor norm }$\alpha $\textit{\ such that for
every Banach spaces }$F,G$\textit{\ we have} 
\begin{equation*}
\mathcal{I}\left( F;G^{\ast }\right) =(F\widehat{\otimes }_{\alpha }G)^{\ast
}.
\end{equation*}%
\textit{Then, there is a Lipschitz cross-norm }$\alpha ^{L}$\textit{\
defined as in }\eqref{3.1}\textit{\ such that}%
\begin{equation*}
\mathcal{I}\circ Lip_{0}\left( X;E^{\ast }\right) =(X\widehat{\boxtimes }%
_{\alpha ^{L}}E)^{\ast }.
\end{equation*}%
If we consider the projective tensor norm $\pi $ and injective tensor norm $%
\varepsilon $, by using \eqref{3.2} and the last corollary, it is not hard
to see that%
\begin{equation*}
Lip_{0}\left( X;E^{\ast }\right) =(X\widehat{\boxtimes }_{\pi ^{L}}E)^{\ast
}=(\mathcal{F}\left( X\right) \widehat{\otimes }_{\pi }E)^{\ast },
\end{equation*}%
and 
\begin{equation*}
\mathfrak{I}_{1}\left( \mathcal{F}\left( X\right) ;E^{\ast }\right) =(X%
\widehat{\boxtimes }_{\varepsilon ^{L}}E)^{\ast }=(\mathcal{F}\left(
X\right) \widehat{\otimes }_{\varepsilon }E)^{\ast }=\mathfrak{I}%
_{1}^{SL}\left( X;E^{\ast }\right) .
\end{equation*}
\end{corollary}

\subsection{Lipschitz Chevet-Saphar norms}

We will consider the Chevet-Saphar tensor norms and we will discuss their
corresponding Lipschitz cross-norms. Let $F,G$ be two Banach spaces. The
Chevet-Saphar norms $g_{p}$\ and $d_{p}$ are defined on a tensor product $%
F\otimes G$ for $1\leq p\leq \infty $ as 
\begin{equation*}
d_{p}\left( u\right) =\inf \left\{ \left\Vert \left( x_{i}\right)
_{i}\right\Vert _{l_{p}^{n,w}\left( F\right) }\left\Vert \left( g_{i}\right)
_{i}\right\Vert _{l_{p^{\ast }}^{n}\left( G\right)
}:u=\sum_{i=1}^{n}x_{i}\otimes g_{i}\right\} 
\end{equation*}%
and 
\begin{equation*}
g_{p}\left( u\right) =\inf \left\{ \left\Vert \left( x_{i}\right)
_{i}\right\Vert _{l_{p^{\ast }}^{n}\left( F\right) }\left\Vert \left(
g_{i}\right) _{i}\right\Vert _{l_{p}^{n,w}\left( G\right)
}:u=\sum_{i=1}^{n}x_{i}\otimes g_{i}\right\} ,
\end{equation*}%
These norms are mainly introduced to study the classes of $p$-summing and
strongly $p$-summing linear operators. The dual spaces of the corresponding
tensor products coincide with these last spaces, i.e.,%
\begin{equation*}
\mathcal{D}_{p}\left( F;G^{\ast }\right) =(F\widehat{\otimes }%
_{g_{p}}G)^{\ast }\text{ and }\Pi _{p}\left( F;G^{\ast }\right) =(F\widehat{%
\otimes }_{d_{p^{\ast }}}G)^{\ast }.%
\vspace{0.5cm}%
\end{equation*}%
We recall that a linear operator $T:F\rightarrow G$ is $p$-summing if there
exists a positive constant $C$ such that for every $x_{1},...,x_{n}\in F$
and $g_{1}^{\ast },...,g_{n}^{\ast }\in G^{\ast }$ we have%
\begin{equation}
\left\vert \sum_{i=1}^{n}\left\langle T\left( x_{i}\right) ,g_{i}^{\ast
}\right\rangle \right\vert \leq Cd_{p}\left( u\right) ,  \label{3.3}
\end{equation}%
where $u=\sum_{i=1}^{n}x_{i}\otimes g_{i}^{\ast }.$ The space $\Pi
_{p}\left( F,G\right) $ stands the Banach space of all $p$-summing linear
operators and%
\begin{equation*}
\left\Vert T\right\Vert _{\Pi _{p}}=\inf \{C\text{, verifying the equality %
\eqref{3.3}}\}\text{\textit{.}}
\end{equation*}%
Moreover, for the definition of strongly $p$-summing linear operators, we
substitute in \eqref{3.3} $d_{p}\left( u\right) $ by $g_{p}\left( u\right) $%
. Again, $\mathcal{D}_{p}\left( F;G\right) $ stands the Banach space of all
strongly $p$-summing linear operators with the norm $\left\Vert T\right\Vert
_{D_{p}}$\textit{.} For more details about these notions see \cite{6,9}$.$
Now, let $X$ be a pointed metric space and $E$ be a Banach space. We define $%
d_{p}^{L},$ the corresponding norm of $d_{p},$ as follows: for every\textit{%
\ }$u=X\boxtimes E$\textit{\ }we have%
\begin{equation*}
d_{p}^{L}\left( u\right) =d_{p}\left( \Phi \left( u\right) \right) =\inf
\left\{ \left\Vert m_{i}\right\Vert _{l_{p}^{n,w}\left( \mathcal{F}\left(
X\right) \right) }\left\Vert \left( e_{i}\right) _{i}\right\Vert
_{l_{p^{\ast }}^{n}\left( E\right) }\right\} ,
\end{equation*}%
where the infimum is taken over all representations of the form $%
\sum_{i=1}^{n}m_{i}\otimes e_{i}\in \mathcal{F}\left( X\right) \otimes E$
such that $\Phi \left( u\right) =\sum_{i=1}^{n}m_{i}\otimes e_{i}.$ By %
\eqref{3.2}, we obtain the following identification for $1\leq p<\infty $%
\begin{equation*}
X\widehat{\boxtimes }_{d_{p}^{L}}E=\mathcal{F}\left( X\right) \widehat{%
\otimes }_{d_{p}}E.
\end{equation*}%
Let $T\in Lip_{0}\left( X;E\right) $ be a Lipschitz operator, the operator $T
$ can be see as a linear functional on $X\boxtimes E^{\ast }$ which its
action on a tensor $u=\sum_{i=1}^{n}\delta _{\left( x_{i},y_{i}\right)
}\boxtimes e_{i}^{\ast }$ is given by%
\begin{equation*}
\left\langle T,u\right\rangle =\sum_{i=1}^{n}\left\langle T\left(
x_{i}\right) -T\left( y_{i}\right) ,e_{i}^{\ast }\right\rangle .%
\vspace{0.5cm}%
\end{equation*}%
Inspired by the definition of $p$-summing linear operators \eqref{3.3}$,$ we
introduce a new definition in the category of Lipschitz mappings.$%
\vspace{0.5cm}%
$

\begin{definition}
Let $1\leq p<\infty .$ Let $X$ be a pointed metric space and $E$ a Banach
space. A Lipschitz operator $T:X\rightarrow E$ is said to be\textit{\
strictly Lipschitz }$p$\textit{-summing} if there exists a positive constant 
$C$ such that for every $x_{i},y_{i}\in X$ and $e_{i}^{\ast }\in E^{\ast }$ $%
\left( 1\leq i\leq n\right) $ we have 
\begin{equation}
\left\vert \sum_{i=1}^{n}\left\langle T\left( x_{i}\right) -T\left(
y_{i}\right) ,e_{i}^{\ast }\right\rangle \right\vert \leq Cd_{p}^{L}\left(
u\right)  \label{3.4}
\end{equation}%
where $u=\sum_{i=1}^{n}\delta _{\left( x_{i},y_{i}\right) }\boxtimes
e_{i}^{\ast }.$ We denote by $\Pi _{p}^{SL}\left( X;E\right) $ the Banach
space of all strictly Lipschitz $p$-summing operators from $X$ into $E$
which its norm $\left\Vert T\right\Vert _{\Pi _{p}^{SL}}$ is the smallest
constant $C$ verifying \eqref{3.4}$.$
\end{definition}

In \cite{4}$,$ the author has studied the class $\Pi _{p}^{L}\left(
X;E\right) $ of Lipschitz $p$-summing operators. He has defined a norm on
the space of molecules $\mathcal{F}\left( X;E\right) $ of which we have the
next duality%
\begin{equation*}
\Pi _{p}^{L}\left( X;E^{\ast }\right) =\mathcal{F}_{cs_{p}}\left( X;E\right)
^{\ast },
\end{equation*}%
where $cs_{p}$ is defined by%
\begin{equation*}
cs_{p}\left( u\right) =\inf \left\{ \left\Vert \delta _{\left(
x_{i},y_{i}\right) }\right\Vert _{l_{p\ast }^{n,w}\left( \mathcal{F}\left(
X\right) \right) }\left\Vert \left( e_{i}\right) _{i}\right\Vert
_{l_{p}^{n}\left( E\right) }\right\} ,
\end{equation*}%
where the infimum is taken over all representations of $u$ of the form $%
u=\sum_{i=1}^{n}\delta _{\left( x_{i},y_{i}\right) }\boxtimes e_{i}\in 
\mathcal{F}\left( X;E\right) .$ Note that the space of molecules $\mathcal{F}%
\left( X;E\right) $\ plays the same role of Lipschitz tensor product $%
X\boxtimes E$ which their norms can be defined on both spaces. Definitions $%
cs_{p}$ and $d_{p}^{L}$ look very similar. However, they do not coincide. In
the definition of $cs_{p}$ we are only using elements in $\mathcal{F}\left(
X\right) $ of the form $\delta _{\left( x,y\right) },$ but in the second
case we have to consider all elements of $\mathcal{F}\left( X\right) .$
Therefore, the infimum in $d_{p}^{L}$ will in general be smaller. It means
that 
\begin{equation*}
\Pi _{p}^{SL}\left( X;E\right) \subset \Pi _{p}^{L}\left( X;E\right) .%
\vspace{0.5cm}%
\end{equation*}

In \cite{15}$,$ we have seen that if the linearization $\widehat{T}$ of $T$\
is $p$-summing then $T$ is Lipschitz $p$-summing, but the converse is not
true in general. In our case, we show that it is true for the concept of
strictly Lipschitz $p$-summing.$%
\vspace{0.5cm}%
$

\begin{theorem}
\label{Theorem 3.5}\textit{Let }$1\leq p<\infty .$ \textit{Let} $X$\textit{\
be a metric space and }$E$\textit{\ be a Banach space. The following
properties are equivalent.}

$\left( 1\right) $ \textit{The Lipschitz operator }$T$\textit{\ belongs to} $%
\Pi_{p}^{SL}\left( X;E\right) .$

$\left( 2\right) $ \textit{The linearization }$\widehat{T}$\textit{\ belongs
to} $\Pi_{p}\left( \mathcal{F}\left( X\right) ;E\right) .$
\end{theorem}

\begin{proof}
$\left( 2\right) \Rightarrow \left( 1\right) :$ Suppose that $\widehat{T}\in
\Pi _{p}\left( \mathcal{F}\left( X\right) ;E\right) .$ Let $x_{i},y_{i}\in X$
and $e_{i}^{\ast }\in E^{\ast }$ $\left( 1\leq i\leq n\right) .$ We put $%
u=\sum_{i=1}^{n}\delta _{\left( x_{i},y_{i}\right) }\boxtimes e_{i}^{\ast }$%
, then 
\begin{align*}
\left\vert \sum_{i=1}^{n}\left\langle T\left( x_{i}\right) -T\left(
x_{i}\right) ,e_{i}^{\ast }\right\rangle \right\vert & =\left\vert
\sum_{i=1}^{n}\left\langle \widehat{T}\left( \delta _{\left(
x_{i},y_{i}\right) }\right) ,e_{i}^{\ast }\right\rangle \right\vert \\
& \leq \left\Vert \widehat{T}\right\Vert _{\Pi _{p}}d_{p}\left( \Phi \left(
u\right) \right) =\left\Vert \widehat{T}\right\Vert _{\Pi
_{p}}d_{p}^{L}\left( u\right) ,
\end{align*}%
hence $T$ is strictly Lipschitz $p$-summing and 
\begin{equation*}
\left\Vert T\right\Vert _{\Pi _{p}^{SL}}\leq \left\Vert \widehat{T}%
\right\Vert _{\Pi _{p}}.
\end{equation*}%
$\left( 1\right) \Rightarrow \left( 2\right) :$ Suppose that $T\in \Pi
_{p}^{SL}\left( X;E\right) .$ Let $m_{i}\in \mathcal{F}\left( X\right) $ $%
(m_{i}=\sum_{j=1}^{k_{i}}\lambda _{i}^{j}\delta _{\left(
x_{i}^{j},y_{i}^{j}\right) }:\lambda _{i}^{j}\in \mathbb{R})$ and $%
e_{i}^{\ast }\in E\left( 1\leq i\leq n\right) $%
\begin{align*}
\left\vert \sum_{i=1}^{n}\left\langle \widehat{T}\left( m_{i}\right)
,e_{i}^{\ast }\right\rangle \right\vert & =\left\vert
\sum_{i=1}^{n}\sum_{j=1}^{k_{i}}\left\langle T\left( x_{i}^{j}\right)
-T\left( y_{i}^{j}\right) ,\lambda _{i}^{j}e_{i}^{\ast }\right\rangle
\right\vert \\
& \leq \left\Vert T\right\Vert _{\Pi _{p}^{SL}}d_{p}^{L}\left( u\right)
=\left\Vert T\right\Vert _{\Pi _{p}^{SL}}d_{p}\left( \Phi \left( u\right)
\right) ,
\end{align*}%
where 
\begin{equation*}
u=\sum_{i=1}^{n}\sum_{j=1}^{k_{i}}\lambda _{i}^{j}\delta _{\left(
x_{i}^{j},y_{i}^{j}\right) }\boxtimes e_{i}^{\ast },
\end{equation*}%
then $\Phi \left( u\right) =\sum_{i=1}^{n}m_{i}\otimes e_{i}^{\ast }.$ So, $%
\widehat{T}$ is $p$-summing and $\left\Vert \widehat{T}\right\Vert _{\Pi
_{p}}\leq \left\Vert T\right\Vert _{\Pi _{p}^{SL}}.$
\end{proof}

As immediate consequences, we have the following results.

\begin{corollary}
\label{Corollary 3.6}\textit{For every pointed metric space }$X$\textit{\
and Banach space }$E$\textit{\ we have }%
\begin{equation*}
\Pi _{p}^{SL}\left( X;E\right) =(X\widehat{\boxtimes }_{d_{p}^{L}}E^{\ast
})^{\ast }=(\mathcal{F}\left( X\right) \widehat{\otimes }_{d_{p}}E^{\ast
})^{\ast }=\Pi _{p}\left( \mathcal{F}\left( X\right) ;E\right) .
\end{equation*}
\end{corollary}

\begin{corollary}
\label{Corollary 3.7}\textit{The next inclusion is strict}%
\begin{equation*}
\Pi _{p}^{SL}\left( \mathbb{R};l_{1}\left( \mathbb{R}\right) \right) \subset
\Pi _{p}^{L}\left( \mathbb{R};l_{1}\left( \mathbb{R}\right) \right) .
\end{equation*}
\end{corollary}

\begin{proof}
\textbf{\ }We know that $\delta_{\mathbb{R}}:\mathbb{R\rightarrow}\mathcal{F}%
\left( \mathbb{R}\right) $ $(\mathcal{F}\left( \mathbb{R}\right)
=l_{1}\left( \mathbb{R}\right) $ in fact$)$ is Lipschitz $p$-summing. Its
linearization is the identity on $\mathcal{F}\left( \mathbb{R}\right) $
which cannot be $p$-summing because $\mathcal{F}\left( \mathbb{R}\right) $
is infinite-dimensional. Hence, $\delta_{\mathbb{R}}$ is not strictly
Lipschitz $p$-summing.
\end{proof}

If $X$ is a Banach space and $T:X\rightarrow E$ is a linear operator, we
have the following result.

\begin{proposition}
\label{Proposition 3.8}\textit{Let }$X,E$\textit{\ be two Banach spaces and }%
$T:X\rightarrow E$\textit{\ be a linear operator. The following properties
are equivalent.}

$\left( 1\right) $ $T$ \textit{is Lipschitz }$p$\textit{-summing.}

$\left( 2\right) $ $T$ \textit{is }$p$\textit{-summing.}

$\left( 3\right) $ $T$ \textit{is strictly Lipschitz }$p$\textit{-summing.}

\textit{In this case we have} 
\begin{equation*}
\left\Vert T\right\Vert _{\Pi_{p}}=\left\Vert T\right\Vert
_{\Pi_{p}^{L}}=\left\Vert T\right\Vert _{\Pi_{p}^{SL}}.
\end{equation*}
\end{proposition}

\begin{proof}
The equivalence $\left( 1\right) \Leftrightarrow \left( 2\right) $ has been
proved by Farmer and Johnson in \cite{11} and we have $\left\Vert
T\right\Vert _{\Pi _{p}}=\left\Vert T\right\Vert _{\Pi _{p}^{L}}.$ Suppose
that $T$ is $p$-summing, by \eqref{1.4} $\widehat{T}=T\circ \beta _{X}$
which is $p$-summing by the ideal property, consequently $T$ is strictly
Lipschitz $p$-summing and we have%
\begin{equation*}
\left\Vert T\right\Vert _{\Pi _{p}^{SL}}=\left\Vert T\circ \beta
_{X}\right\Vert _{\Pi _{p}}\leq \left\Vert T\right\Vert _{\Pi _{p}}.
\end{equation*}%
The last implication is immediate with $\left\Vert T\right\Vert _{\Pi
_{p}^{L}}\leq \left\Vert T\right\Vert _{\Pi _{p}^{SL}}$.
\end{proof}

In the linear case, every $p$-integral operator linear is $p$-summing. Then,
by Proposition \ref{Proposition 2.12} and Theorem \ref{Theorem 3.5} we
conclude that every strongly Lipschitz $p$-integral operator is strictly
Lipschitz $p$-summing. In the next results, we give some coincidence
situations as in the linear case, for the proof we use the linearization of
both classes and the linear results given in \cite[p 99]{9}.

\begin{corollary}
\label{Corollary 3.9}\textit{Let }$1\leq p<\infty .$\textit{\ Let }$X$%
\textit{\ be a pointed metric space and }$E$\textit{\ be an injective Banach
space. Then}%
\begin{equation*}
\Pi _{p}^{SL}\left( X;E\right) =\mathfrak{I}_{p}^{SL}\left( X;E\right) ,
\end{equation*}%
\textit{with equality of norms.}
\end{corollary}

\begin{corollary}
\label{Corollary 3.10}\textit{Let }$X$\textit{\ be a pointed metric space
and }$E$\textit{\ be a Banach space. Then}%
\begin{equation*}
\Pi _{2}^{SL}\left( X;E\right) =\mathfrak{I}_{2}^{SL}\left( X;E\right) ,
\end{equation*}%
\textit{with equality of norms.}
\end{corollary}

\begin{corollary}
\label{Corollary 3.11}\textit{If }$E$\textit{\ is a subspace of an }$%
\mathcal{L}_{p}$-space, $1\leq p\leq 2$.\textit{\ Then for every pointed
metric space }$X$%
\begin{equation*}
\Pi _{q}^{SL}\left( X;E\right) =\mathfrak{I}_{q}^{SL}\left( X;E\right) =%
\mathfrak{I}_{2}^{SL}\left( X;E\right) ,
\end{equation*}%
\textit{for all }$2\leq q<\infty $\textit{.}
\end{corollary}

We next show a Lipschitz version of well know characterization of an $%
\mathcal{L}_{\infty }$-space which said that a Banach space $X$ is an $%
\mathcal{L}_{\infty }$-space if and only if for every Banach space $E$ and $%
1 $-summing linear operator $T:X\rightarrow E$, $T$ is $1$-integral, see 
\cite[Corollary 6.24]{9} for more details about this characterization.

\begin{theorem}
\label{Theorem 3.12}\textit{Let }$X$\textit{\ be a pointed metric space. The
following properties are equivalent.}

$\left( 1\right) $ \textit{The space} $\mathcal{F}\left( X\right) $ \textit{%
is an} $\mathcal{L}_{\infty}$\textit{-space}.

$\left( 2\right) $ \textit{For all Banach space }$E$\textit{\ we have} $%
\Pi_{1}^{SL}\left( X;E\right) =\mathfrak{I}_{1}^{SL}\left( X;E\right) .$
\end{theorem}

\begin{proof}
$\left( 1\right) \Rightarrow\left( 2\right) :$ Suppose that $\mathcal{F}%
\left( X\right) $ is an $\mathcal{L}_{\infty}$-space. Let $E$ be a Banach
space and $T\in\Pi_{1}^{SL}\left( X;E\right) $, then $\widehat{T}:\mathcal{F}%
\left( X\right) \rightarrow E$ is $1$-summing. By the characterization of an 
$\mathcal{L}_{\infty}$-space, $\widehat{T}$ is $1$-integral. Consequently, $%
T $ is in $\mathfrak{I}_{1}^{SL}\left( X;E\right) .$

$\left( 2\right) \Rightarrow\left( 1\right) :$ Let $v:\mathcal{F}\left(
X\right) \rightarrow E$ be $1$-summing linear operator. It suffices to see
that $v$ is $1$-integral. Put $T=v\circ\delta_{X}.$ Then $T$ is strictly
Lipschitz $1$-summing and $\widehat{T}=v$. By $\left( 2\right) ,$ $T$ is
strongly Lipschitz $1$-integral and then its linearization is $1$-integral.
\end{proof}

We finish this section by discussing the Lipschitz tensor norm associated to
Chevet-Saphar norm $g_{p}$. For every\textit{\ }$u=\sum_{i=1}^{n}\delta
_{\left( x_{i},y_{i}\right) }\boxtimes e_{i}\in X\boxtimes E$\textit{\ }we
have%
\begin{equation*}
g_{p}^{L}\left( u\right) =g_{p}\left( \Phi \left( u\right) \right) .
\end{equation*}%
Then%
\begin{equation*}
(X\widehat{\boxtimes }_{g_{p}^{L}}E)^{\ast }=\left( \mathcal{F}\left(
X\right) \widehat{\otimes }_{g_{p}}E\right) ^{\ast }.
\end{equation*}%
Proposition 3.1 in \cite{15} gives that there is an equivalence between a
Lipschitz mapping $T$ and its linearization $\widehat{T}$ for the concept of
strongly $p$-summing. Moreover, we have%
\begin{equation*}
\mathcal{F}_{\mu _{p}}\left( X;E\right) ^{\ast }=\mathcal{D}_{p}^{L}\left(
X;E^{\ast }\right) ,
\end{equation*}%
where the norm $\mu _{p}$ is defined as follows%
\begin{equation*}
\mu _{p}\left( u\right) =\inf \left\{ \left\Vert \delta _{\left(
x_{i},y_{i}\right) }\right\Vert _{l_{p\ast }^{n}\left( \mathcal{F}\left(
X\right) \right) }\left\Vert \left( e_{i}\right) _{i}\right\Vert
_{l_{p}^{n,w}\left( E\right) }:u=\sum_{i=1}^{n}\delta _{\left(
x_{i},y_{i}\right) }e_{i}\right\} .
\end{equation*}%
Combining with Theorem \ref{Theorem 2.1}, we obtain the following
identification.

\begin{theorem}
\label{Theorem 3.13}\textit{Let }$X$\textit{\ be a metric space and }$E$%
\textit{\ be a Banach space. We have}%
\begin{equation*}
(X\widehat{\boxtimes }_{g_{p}^{L}}E)^{\ast }=(\mathcal{F}\left( X\right) 
\widehat{\otimes }_{g_{p}}E)^{\ast }=\mathcal{D}_{p}^{L}\left( X;E^{\ast
}\right) =\mathcal{D}_{p}\left( \mathcal{F}\left( X\right) ;E^{\ast }\right)
.
\end{equation*}
\end{theorem}

\begin{corollary}
\label{Corollary 3.14}\textit{The norms }$g_{p}^{L}$\textit{\ and }$\mu _{p}$%
\textit{\ are the same.}
\end{corollary}

\begin{proof}
Let $u\in X\boxtimes E$. By the definition of $g_{p}^{L}$ we have%
\begin{equation*}
g_{p}^{L}\left( u\right) \leq \mu _{p}\left( u\right) .
\end{equation*}%
On the other hand,%
\begin{align*}
\mu _{p}\left( u\right) & =\sup_{T\in B_{\mathcal{F}_{\mu _{p}}\left(
X;E\right) ^{\ast }}}\left\vert \left\langle T,u\right\rangle \right\vert
=\sup_{\widehat{T}\in B_{\mathcal{D}_{p}\left( \mathcal{F}\left( X\right)
;E^{\ast }\right) }}\left\vert \left\langle \widehat{T},u\right\rangle
\right\vert \\
& =\sup_{\widehat{T}\in B_{\mathcal{D}_{p}\left( \mathcal{F}\left( X\right)
;E^{\ast }\right) }}\left\vert \sum_{i=1}^{n}\left\langle \widehat{T}\left(
\delta _{\left( x_{i},y_{i}\right) }\right) ,e_{i}\right\rangle \right\vert
\\
& \leq \sup_{\widehat{T}\in B_{\mathcal{D}_{p}\left( \mathcal{F}\left(
X\right) ;E^{\ast }\right) }}\left\Vert \widehat{T}\right\Vert _{\mathcal{D}%
_{p}}g_{p}(\sum_{i=1}^{n}\delta _{\left( x_{i},y_{i}\right) }\otimes e_{i})
\\
& \leq g_{p}\left( \Phi \left( u\right) \right) =g_{p}^{L}\left( u\right) .
\end{align*}
\end{proof}

\end{document}